\documentclass[11pt]{article}
\usepackage{fullpage}
\usepackage[left=1in,top=1in,right=1in,bottom=1in,headheight=3ex,headsep=3ex]{geometry}
\usepackage{graphicx}
\usepackage{enumitem}
\usepackage{amssymb}
\usepackage{amsthm}
\usepackage{mdframed}
\usepackage[skip=2pt,font=scriptsize]{caption}

\mdfdefinestyle{E}{%
   linecolor=gray!30!white,
  outerlinewidth=0pt,
  roundcorner=20pt,
  innertopmargin=\baselineskip,
  innerbottommargin=\baselineskip,
  innerrightmargin=20pt,
  innerleftmargin=10pt,
  backgroundcolor=gray!30!white}
 \mdfdefinestyle{P}{%
    linecolor=gray!5!white,
 frametitlefont=\color{black} \bf ,
 frametitlerulewidth=0pt
  outerlinewidth=0pt,
  innertopmargin=0pt,
  innerbottommargin=0pt,
  innerrightmargin=10pt,
  innerleftmargin=10pt,
  backgroundcolor=gray!3!white}
 \mdtheorem[style=P]{Theorem}[proposition]{Theorem}
 \mdtheorem[style=P]{Proposition}[proposition]{Proposition}
 \mdtheorem[style=P]{Note}[proposition]{Note}
 \mdtheorem[style=P]{Definition}[proposition]{Definition}
  \mdtheorem[style=P]{Notations}[proposition]{Notations}
\title{On the domain of convergence of spherical harmonic expansions}
\author{O. Costin$^1$, R. D. Costin$^1$, C. Ogle$^1$, M. Bevis$^2$}
\usepackage[sc]{mathpazo}
\usepackage[T1]{fontenc}
\usepackage{url}
\usepackage{hyperref}\hypersetup{
  colorlinks=true, 
  linktoc=all,  
}

\usepackage{amsmath}
 \newtheorem{proposition}{proposition}
 \newtheorem*{Assumptions*}{Assumptions}
\newtheorem{Lemma}[proposition]{Lemma}

\newtheorem*{Examples*}{Examples}

\newtheorem*{Corollary*}{Corollary}

\usepackage{array}
\usepackage[dvipsnames]{xcolor}
\usepackage{soul}
\sethlcolor{SkyBlue}
\usepackage{hyperref}

\definecolor{clemsonorange}{HTML}{EA6A20}
\hypersetup{colorlinks,breaklinks,linkcolor=clemsonorange,urlcolor=clemsonorange,anchorcolor=clemsonorange,citecolor=black}

\def\RR{\mathbb{R}}

\def\CC{\mathbb{C}}
\def\NN{\mathbb{N}}

\def\TT{\mathbb{T}}
\def\DD{\mathbb{D}}

\def\<{\langle}
\def\({\left (}
 \def\){\right)}
\def\epsilon{\varepsilon}
\def\phi{\varphi}
 
\def\le{\leqslant}
\def\ge{\geqslant}

\def\>{\rangle}

\def\0{\emptyset}

\def\r{\textcolor{red}}
\def\r{\textcolor{black}}
\date{$^1$ Department of Mathematics, The Ohio State University\\$^2$ Division of Geodetic Science, The Ohio State University\\ \today}
\begin{document}

\maketitle
\begin{abstract}
  Spherical harmonic expansions (SHEs) play an important role in most of the physical sciences, especially in physical geodesy. Despite many decades of investigation, the large order behavior of the SHE coefficients, and the precise domain of convergence for these expansions, have remained open questions. These questions are settled in the present paper for generic planets, whose shape (topography) may include many local peaks, but just one globally highest peak. We show that regardless of the smoothness of the density and topography, short of outright analyticity, the spherical harmonic expansion of the gravitational potential converges exactly in the closure of the exterior of the Brillouin sphere\footnote{The smallest sphere around the center of mass of the planet containing the planet in its interior, see Figure \ref{fig:1}.}, and convergence below the Brillouin sphere occurs with probability zero. More precisely, such over-convergence occurs on zero measure sets in the space of parameters. A related result is that, in a natural Banach space, SHE convergence of the potential below the Brillouin sphere occurs for potential functions in a subspace of infinite codimension (while any positive codimension already implies occurrence of probability zero). Provided a certain limit in Fourier space exists, we find the leading order asymptotic behavior of the coefficients of SHEs.

We go further by finding a necessary and sufficient condition for convergence below the Brillouin sphere, which requires a form of analyticity at the highest peak, which would not hold for a realistic celestial body.   Namely, a longitudinal average of the harmonic measure on the Brillouin sphere would have to be real-analytic at the point of contact with the boundary of the planet. It turns out that only a small neighborhood of the peak is involved in this condition.

\end{abstract}

\section{Introduction and overview of the results}\label{intro}

Spherical harmonic expansions (SHEs) play an important role in many branches of physics, geophysics and planetary science, especially in physical geodesy \cite{Kellog},\,\cite{Heiskanen},\,\cite{Pavlis}. The large order behavior of the SHE coefficients, and the precise domain of convergence of the expansions have been open questions in mathematical geodesy at least since the 1960's \cite{Krarup},\,\cite{Moritz},\,\cite{Jekeli},\,\cite{Wang},\,\cite{Hu},\,\cite{Hirt}.

These questions are settled in the present paper for generic model planets with possibly many local peaks but a unique globally highest one, and with various degrees of regularity of the topography  and mass density. We find that the decay of the coefficients is faster for smoother planets\footnote{Except for pathological shapes  where the highest peak is a sharp cusp.}. Our analysis would easily extend to a finite number of peaks of equal height, but that is  non-generic and would unnecessarily complicate the calculations.

Theorem \ref{Thm1} shows that the domain of convergence only depends on the regularity properties of the surface of the planet and of the mass density. Perhaps remarkably, smoothness only matters  in an arbitrarily small neighborhood of the highest peak of the surface,  while the rest of the features of the planet are immaterial. The domain of convergence always contains the closed exterior of the Brillouin sphere (the smallest sphere around the center of mass of the planet containing the planet in its interior, see Figure \ref{fig:1}) and, except for zero measure sets in the parameter space, contains no point below it. More precisely, in a natural Banach space, SHE converge at no point below the Brillouin sphere except for a set of infinite codimension (hence meagre). Provided a certain limit in Fourier space exists see \eqref{eq:regcond}, we find the leading order asymptotic behavior of the coefficients of the SHE (see \eqref{eq:asympt-beh}). The Banach space we use stems from a generalization of this Fourier condition.

Theorem \ref{Thm1} assumes some minimal regularity for the topography; lower regularity classes are treated in Theorem \ref{Thm3}.

 Using potential theory and very recent methods of resurgent analysis \cite{CostinDunne}, in Theorem \ref{Thm2} we find a necessary and sufficient condition of convergence below the Brillouin sphere. This condition is that a longitudinal average of the harmonic measure on the Brillouin sphere has to be real-analytic\footnote{Of course, such a feature cannot be expected from a realistic celestial body\label{F1}.}. Only an (arbitrarily small) neighborhood of the highest peak is involved in this condition.

In the parallel paper \cite{ocb}, using elementary topological methods but assuming only continuity, we establish a (weaker) density result: that for any $\epsilon>0$ there is a dense set of planets for which convergence does not extend by a distance exceeding $\epsilon$ below the Brillouin sphere.

In  upcoming work we will use the asymptotic formulas derived here to estimate the optimal place to truncate the SHE on or near the Brillouin sphere, that is, the order of the SHE polynomial that ensures maximal accuracy. There, we will also address the practical question of how large a neighborhood of the peak one has to consider in order that the regularity features affect a fixed number of coefficients. Using the results in \cite{CostinDunne} we will provide optimal methods of downward continuation of the potential given by a truncated SHE from the Brillouin sphere down to the topography.

 In Section \ref{Further} we provide a physical interpretation of our results and methods.

\begin{figure}
  \centering
  \includegraphics[scale=0.7]{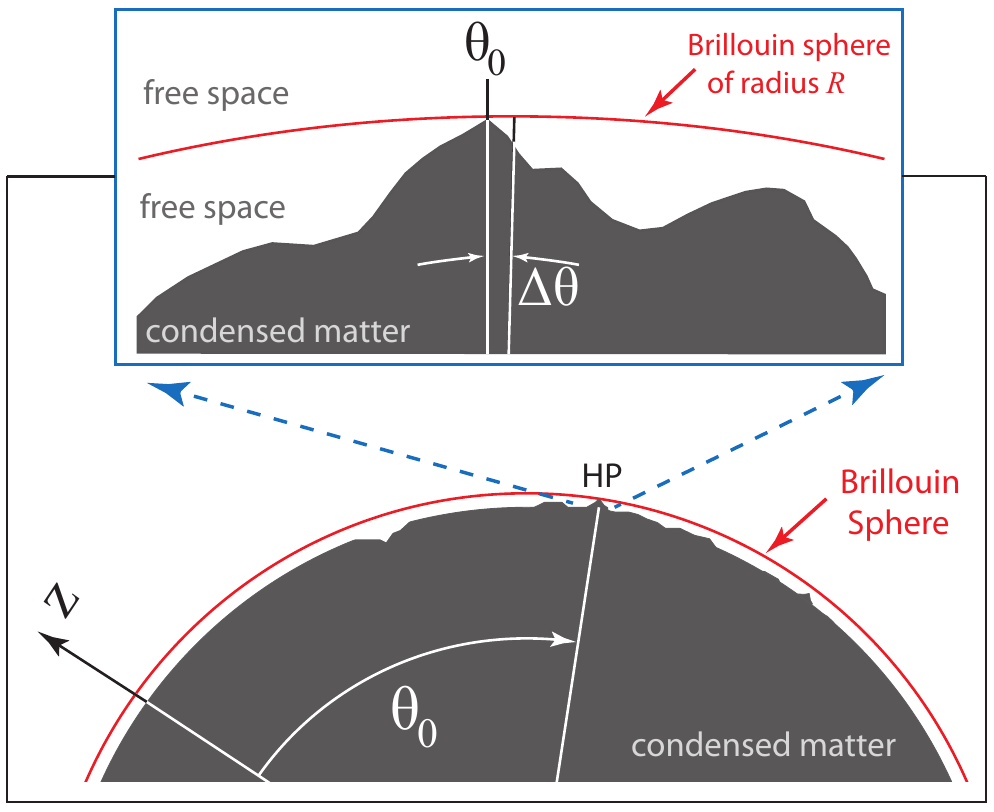}
  \caption{A generic planet with exactly one global highest peak (HP) at latitude $\theta_0$ (enlarged part). The radius of convergence of the SHE only depends on the regularity of the surface shape and of the mass density in an arbitrarily small neighborhood of the peak. The the red circle marks the Brillouin sphere (the smallest sphere around the center of mass of the planet containing the planet in its interior).} 
  \label{fig:1}
\end{figure}

\section{Notations} 
Let $P$ be a planet and fix an observation point outside $P$. Choose a polar representation of the points in the planet so that the observation point is on the $Oz$ axis, at distance $z$ to the origin. Let $R=\{\max r:(r,\theta,\lambda)\in P\}$ be the Brillouin radius. (Here we use the common convention in applications of SHE,  where the Eulerian angle corresponding to longitude is denoted by $\lambda$.)

We also make the common assumption that the influence of the atmospheric mass is negligible.

The gravitational potential at $(0,0,z)$, $z>R$ can be expanded as a power series in $z$ (using dominated convergence and the generating function of Legendre polynomials $P_n$), as
\begin{equation}
  \label{eq:potential-planet}
 {-G} \int_P\frac{\rho(r,\theta,\lambda)r^2\sin\theta\, dr\, d\theta\, d\lambda}{\sqrt{ r^2-2{r}{z}\cos\theta+z^2}}= \frac{-G}z \int_P\frac{\rho(r,\theta,\lambda)r^2\sin\theta\, dr\, d\theta\, d\lambda}{\sqrt{(\frac{r}{z})^2-2\frac{r}{z}\cos\theta+1}}=z^{-1}\sum_{n\ge 0}C_n z^{-n}
\end{equation}
where
\begin{equation}
  \label{eq:eqcoef}
  C_n=-G\int_Pr^{n+2}\rho(r,\theta,\lambda)\,  \sin\theta \, P_n(\cos\theta)\, dr\, d\theta\, d\lambda
\end{equation}
is the $n$-th coefficient of the SHE  of the potential. Integrating first in $\lambda$, and assuming for simplicity (which is also the generic case), that
for each $r$ and $\theta$ the domain of integration in $\lambda$ is an interval $[\lambda_m(r,\theta),\lambda_M(r,\theta)] $,  we get 
\begin{equation}
  \label{eq:eqcoef2}
    C_n=\int_0^\pi  \sin\theta P_n(\cos\theta)\, d\theta\int_{I_\theta}\, r^{n+2}v(r,\theta)\, dr
  \end{equation}
  where
  \begin{equation}
    \label{defv}
     v(r,\theta)=-G\int_{\lambda_m(r,\theta)}^{\lambda_M(r,\theta)}\rho(r,\theta,\lambda) d\lambda
  \end{equation}
 and $I_\theta=\{r|(r,\theta,\lambda)\in P\text{ for some }\lambda\in[0,2\pi)\}$. For the same reason as above we assume $I_\theta$ is an interval, $$I_\theta=[r_m(\theta),r_M(\theta)]$$
Define
\begin{equation}
  \label{defg}
  g(\theta)=\sqrt{\sin\theta}\ v(r_M(\theta),\theta)
 \end{equation}
In the following, for two sequences $\{f_1(n)\}_{n\in\NN}$ and $\{f_2(n)\}_{n\in\NN}$  we write 
$$f_1\asymp f_2 \ \ \ \ \text{if\ \ \ \ }\lim_{n\to\infty} \frac{f_1(n)}{f_2(n)}=1$$ 
We write $\Re(\cdot)$ to denote the real part of a complex number or complex-valued function. For the Fourier transform we use the convention 
 $\mathcal{F}f(k)=\hat{f}(k)=:\frac 1{\sqrt{2\pi}}\int_{-\infty}^\infty e^{-ikx}f(x)\,dx$.
\subsection{A genericity  assumption on the planet}\label{asnot}

Assume that there is a unique angle $\theta_0$ of absolute maximum of $r_M(\theta)$, where $r_M(\theta_0)=R$, and that $\theta_0\in (0,\pi)$. 

Define $F$ by
\begin{equation}
  \label{eq:defF}
  e^{-F(\theta)}=\frac{r_M(\theta)}{R}
\end{equation}
We assume there is a neighborhood $I_0=(\theta_0-\delta, \theta_0+\delta)$  such that  outside it $F(\theta)>\delta_1>0$. By the above, we have $F(\theta_0)=0$, $F(\theta)>0$ for $\theta\ne \theta_0$, and $F$ has a minimum at $\theta_0$. 

%%%%%%%%%%%%%%%%%%%%%%%%%%%%%%%%%%%%%%%%%%%%
%%%%%%%%%%%%%%%%%%%%%%%%%%%%%%%%%%%%%%%%%%%%

  \section{Results}

\subsection{Main theorem for Fourier-based local regularity}
To measure local regularity of a function $f$ at the point $\theta_0$ we employ $C^\infty$ cutoff functions $\phi$ near $\theta_0$ and look at the decay of $\widehat{\phi f}$ for large $|k|$.

  \begin{Theorem}\label{Thm1}
  
 We assume that near $\theta_0$
 $$F(\theta)=c(\theta-\theta_0)^2+h(\theta),\ \ \ \ \text{with\ }h(\theta)=O(\theta-\theta_0)^\beta\ \ (\theta\to\theta_0),\ \beta>2$$
 Let $\phi$ be a cutoff function near $\theta_0$: $\phi\in C^\infty(\RR)$, $\phi(\theta)=1$ for $|\theta-\theta_0|\le\epsilon$ and $\phi(\theta)=0$ for $|\theta-\theta_0|\ge 2\epsilon$, where $\epsilon>0$ is small enough.
 Assume
 \begin{equation}\label{Linf}
 (1+|k|)^{\beta_0} \mathcal{F}(\phi g)\in L^\infty(\RR)\ \ \ \ \text{and}\ \ \ (1+|k|)^{\beta_1} \mathcal{F}(\phi h g)\in L^\infty(\RR), \ \ \text{for\ some\ }\beta_0>1,\beta_1>2
 \end{equation}
\begin{enumerate}
\item  Then
 \begin{equation}
   \label{eq:limsup}
   \limsup_{n\in\NN} R^{-n}n^{\tfrac32+\beta_m}|C_n|\ne 0
 \end{equation}
where $\beta_m=\min\{\beta_0,\beta_1-1\}$, and the domain of convergence of the SHE \eqref{eq:potential-planet} is exactly $\{z\in \CC:|z|\ge R\}$,  apart from exceptional functions $g,h$: they belong to a set of infinite codimension (hence meagre, or negligible) in a natural Banach space.
 
\item Assume further that
\begin{equation}
  \label{eq:regcond}
 \lim_{k\to -\infty} (-k)^{\beta_0} \mathcal{F}(\phi g)=a_0\ne 0 \ \ \ \text{and\ } \ \ \lim_{k\to -\infty} (-k)^{\beta_1} \mathcal{F}(\phi hg)=a_1\ne 0
\end{equation}
for some $\beta_0>1,\beta_1>2$ with $a_0,a_1\in \CC$.

Then the SHE coefficients $C_n$ have the asymptotic behavior
\begin{equation}
      \label{eq:asympt-beh}
   C_n= 2\frac{R^{n+3}}{n^{3/2}}\Re \left[e^{-i\pi/4} e^{i(n+1/2)\theta_0}\left(  \frac{a_0 }{n^{\beta_0}}-  \frac{a_1 }{n^{\beta_1-1}}\right)\right](1+o(1))\ \text{as $n\to\infty$}
    \end{equation}
    unless $\beta_0=\beta_1-1$ and $a_0=a_1$. 
    
    The limits $a_0,\,a_1$ are independent of the size of the support of the cutoff functions in a neighborhood of $\theta_0$.
    \end{enumerate}

  \end{Theorem}
  The proof of Theorem \eqref{Thm1} is given in \S\ref{PfThm1}.
\begin{Note*}
 The infinite codimension of the exceptional set is intuitively clear from the fact that the limsup in \eqref{eq:limsup} has to be zero for all $\beta$. Furthermore, that would only ensure that $F$ and $g$ are locally $C^\infty$, and  convergence below the Brillouin sphere would still not be guaranteed; see \S\ref{S:necsuf}.
\end{Note*}

\subsection{A necessary and sufficient condition of  convergence of SHE for some $z_0<R$}\label{S:necsuf}
A necessary and sufficient condition of convergence for some $z_0$ with $|z_0|<R$ is given in the following theorem.

Given a volume $\mathcal{V}$ bounded by the surface $\Gamma$, and a mass distribution $\rho(r,\theta,\lambda)$ in $L^\infty$, the balayage method \cite{bal1,bal2,bal3,bal4,bal5} provides a surface mass density function $\sigma(\theta,\lambda)$  on $\Gamma$ such that the gravitational potential produced by $\tilde{\rho}$ on $\Gamma$ equals, outside $\Gamma$, the potential produced by $\rho$ on $\mathcal{V}$. For the next theorem we need to take $\Gamma$ to be the Brillouin sphere and {\em extend the planet $P$} to the whole interior of the Brillouin sphere, by assigning zero density to any point between $P$ and the Brillouin sphere. This evidently has no effect on the potential outside $P$. See also \S\ref{S:phys1}.   Choosing radial units so that $R=1$, by balayage, the gravitational potential of $P$ can be written as (with $S^2$ the unit sphere)
\begin{equation}
  \label{eq:balayage1}
V(z)=-G\int_{S^2}\frac{\sigma(\theta,\lambda)\sin\theta d\lambda d\theta}{\sqrt{z^2-2 z \cos\theta+1}}=-\int_0^\pi \frac{\mu(\cos \theta )\sin \theta  d\theta}{\sqrt{z^2-2 z \cos \theta+1}}=\int_{-1}^1\frac{\mu(x )  dx}{\sqrt{z^2-2 zx +1}}
\end{equation}
where we integrated first in $\lambda$, denoted 
\begin{equation}
  \label{eq:oneint}
  \mu(\cos \theta )=G\int_0^{2\pi}\sigma(\theta,\lambda)d\lambda
\end{equation}
and changed the variable $\cos\theta=x$.\footnote{Since at $\theta=0,\pi$ the latitude is not defined, these are excluded. We also exclude $\theta=\pi/2$, a special case, for brevity of the arguments.} We further write \eqref{eq:balayage1} as 
\begin{equation}
  \label{eq:reproc1}
V(z)=\frac{1}{\sqrt{z^2+1}}\int_{-1}^1 \frac{\mu(x) dx}{\sqrt{1- p x}}:= \frac{1}{\sqrt{z^2+1}}Q(p), \text{ where }p=\frac{2z}{z^2+1}, \ Q(p)=\int_{-1}^1\frac{\mu(x) dx}{\sqrt{1- p x}}
 \end{equation}

\begin{Theorem}\label{Thm2}
  Assume $\mu$ is H\"older continuous. The SHE  converges at some point $z_0<1$ iff $\mu(\cos\theta)$ is real-analytic at any $\theta\in (0,\tfrac\pi 2)\cup (\tfrac\pi 2,\pi)$, including at the point $\theta_0$ where the Brillouin sphere touches the planet provided that $\theta_0\ne 0,\tfrac\pi 2,\pi$.
\end{Theorem}
The proof of Theorem \ref{Thm2} is given in \S\ref{ST2}.
\begin{Note*}
  \begin{enumerate}
    
  \item The case $\theta_0=0$ could have been included via the introduction of some additional machinery (see the comment after the proof of Theorem \ref{Thm2}); however the genericity assumption allows us to safely exclude it.

  \item If the SHE converges at $z_0$, then (by Abel's theorem) it converges to an analytic function in the domain $\{z\in\CC:|z|>z_0\}$, and we employ the term analyticity in this sense.
  \item  Existence of analytic continuation of the balayage measure inside $P$ would ensure that the criterion of the theorem would be met. Though only exceptional planets would have this property, it is a property which is {\bf not} automatically prevented by the fact\, inside $P$, the potential does not satisfy Laplace's equation anymore, but Poisson's equation. Indeed, if $P$ is a ball with radially symmetric density, the potential outside $P$ is equivalent to that of a point at the center of $P$. Clearly, analytic continuation exists through $P$ except at it's center. The way out of this apparent paradox is that the potential obtained by analytic continuation inside $P$ is simply {\bf not} that of $P$.

  \item The balayage measure for a ball has an explicit formula, as the Green's function for a ball is also explicit (obtained by the  method of images, \cite{Evans}); see \S\ref{S:Greenf}. In the explicit formula it is manifest that, if one chooses a ball of strictly larger radius than the Brillouin radius, then the balayage measure on the larger sphere would have analytic continuation down to the Brillouin sphere. The same formula  also shows that analyticity is ensured away from an arbitrarily small neighborhood of the highest peak.
  \end{enumerate}
\end{Note*}

\subsection{Results in lower regularity}\label{SLowerReg}

Precise asymptotics in regularity lower than one derivative would be best done with Fourier analysis, as in Theorem\,\ref{Thm1}. This would require a more substantial modification of the arguments and will be done elsewhere. The expected result is still \eqref{eq:asympt-beh} with $\beta_0,\beta_1$ less than $1$. The model below is generally unrealistic, but exhibits interesting phenomena, see Note\,\ref{NLowerReg}.

Denote $F(x+\theta_0)=\tilde{F}(x)$. Then $\tilde{F}(0)=0$, and $\tilde{F}(x)>0$  for $x\ne 0$.
We assume  that $F$ and $g$ are continuous, and have the same regularity.
We also assume that $\tilde{F}$ has exactly $\ell\in\{0,1\}$ derivatives at $x=0$, having one of the following behaviors at $x=0$.

For $\ell=0$ we assume that in a small neighborhood of zero

\begin{equation}
\label{assumeF}
\tilde{F}(x)=\left\{ \begin{array}{l l}
a_-|x|^\alpha+O\left(|x|^{\beta}\right) &\text{for\ } x<0\\
a_+|x|^\alpha+O\left(|x|^{\beta}\right) &\text{for\ } x>0
\end{array}\right. ,\ \ \ \alpha\in(0,1],\ \beta>\alpha>0
\end{equation}
Note that the condition that $F$ has a minimum at $\theta_0$ implies  $a_\pm>0$.

For $\ell=1$, we assume that for $|x|<\delta$
\begin{equation}
\label{assumeFb}
\tilde{F}(x)=\left\{ \begin{array}{l l}
a_- |x|^{\alpha} &\text{for\ } -\delta <x<0\\
a_+ \, |x|^{\alpha} &\text{for\ } 0<x<\delta
\end{array}\right. ,\ \ \ \alpha\in(1,2]
\end{equation}
with  $a_\pm>0$. Note that $\tilde{F'}(0)=F(0)=0$.

For each $\ell$ the assumptions on $g$ are similar.

  \begin{Theorem}\label{Thm3} 
    Assume $F$ and $g$ satisfy the assumptions of Section\,\ref{asnot}.

 1.  (i) Assume that,  additionally, $\tilde{F}$ satisfies \eqref{assumeF} and the expansion  of $g$ at $\theta_0$, with remainder, has one of the following forms
 
 a) $g\in C^k$, $k\ge 1$, of the form
  \begin{equation}\label{asumgka}
 g(\theta)=g_k(\theta-\theta_0)^k\, (1+h(\theta)) ,\ g_k\ne 0
 \end{equation}
  with $h(\theta_0)=0$, $h$ continuous\footnote{We note that $g(\theta_0)=0$, see \eqref{defv}, \eqref{defg}.}, or
  
  b) $g$ has a  cusp at $\theta_0$:  for some $k\ge 1$ (not necessarily integer),  
   \begin{equation}\label{asumgkb}
 g(\theta)=(1+h(\theta))\times \  \left\{ 
 \begin{array}{ll}
 g_{+}|\theta-\theta_0|^{k} & \text{for\ } \theta>\theta_0 \\
  g_{-}|\theta_0-\theta|^{k}& \text{for\ } \theta<\theta_0
  \end{array}\right.
  \end{equation}
  with $h(\theta_0)=0$, $h$ continuous, and $g_+,g_-$ not both zero.

 If $g$ has the form \eqref{asumgka} then for $\alpha\in(0,1)$ we have
 \begin{equation}
    \label{eq:asymp10a}
    C_n\asymp\frac{R^{n+3}}{n^{3/2+({k}+1)/\alpha}}\frac{\sqrt{2}\, \Gamma\left(\frac{{k}+1}\alpha\right)\, g_{{k}}}{\alpha\sqrt{\pi}}\Re \left\{e^{-i\pi/4}  e^{i(n+1/2)\theta_0}   \, \left[  a_+^{-\frac {k+1}\alpha} + (-1)^k a_-^{-\frac {k+1}\alpha} \right]\right\}
  \end{equation}
while for $\alpha=1$ we have
  \begin{equation}
    \label{eq:asymp10aa1}
    C_n\asymp\frac{R^{n+3}}{n^{3/2+{k}+1}}\frac{\sqrt{2}\, \Gamma\left({{k}+1}\right)\, g_{{k}}}{\sqrt{\pi}}\Re \left\{e^{-i\pi/4}  e^{i(n+1/2)\theta_0}   \, \left[  (a_+-i)^{-( {k+1})} + (-1)^k (a_-+i)^{-({k+1}) }\right]\right\}
  \end{equation}
 \r{(assuming the right sides of \eqref{eq:asymp10a}, \eqref{eq:asymp10aa1} do dot vanish).}
 
 If $g$ has the form \eqref{asumgkb} then:
 for $\alpha\in(0,1)$
 \begin{equation}
    \label{eq:asymp10ab}
    C_n\asymp\frac{R^{n+3}}{n^{3/2+(k+1)/\alpha}}\frac{\sqrt{2}\, \Gamma\left(\frac{{k}+1}\alpha\right)}{\alpha\sqrt{\pi}}\Re \left[e^{-i\pi/4}  e^{i(n+1/2)\theta_0}   \, \left(  g_+a_+^{-\frac {k+1}\alpha} + g_- a_-^{-\frac {k+1}\alpha} \right)\right]
  \end{equation}
while for $\alpha=1$ we have
  \begin{equation}
    \label{eq:asymp10ba1}
    C_n\asymp\frac{R^{n+3}}{n^{3/2+{k}+1}}\frac{\sqrt{2}\, \Gamma\left({{k}+1}\right)}{\sqrt{\pi}}\Re \left\{e^{-i\pi/4}  e^{i(n+1/2)\theta_0}   \, \left[ g_+ (a_+-i)^{-( {k+1})} + g_- (a_-+i)^{-({k+1}) }\right]\right\}
  \end{equation}
 \r{(assuming the right sides of \eqref{eq:asymp10ab}, \eqref{eq:asymp10ba1} do dot vanish).}

  (ii) If, additionally, $\tilde{F}$ satisfies \eqref{assumeFb} and $\tilde{g}(x):=g(x+\theta_0)$  has a similar regularity at zero:
   \begin{equation}
  \label{assumeg1}
\tilde{g}(x)=g_1x+\left\{ \begin{array}{l l}
g_-|x|^\alpha &\text{for\ } -\delta<x<0\\
g_+|x|^\alpha &\text{for\ } 0<x<\delta
\end{array}\right.
\end{equation}
   then, for $\alpha\in (1,2]$
  \begin{multline}
  \label{eq17}
 C_n\asymp\frac{R^{n+3}}{n^{\frac 32+\alpha+1}}\frac{\sqrt{2}\, \Gamma(\alpha+1)}{\sqrt{\pi}}\times \\
 \Re \left\{ e^{-i\pi/4}  e^{i(n+1/2)\theta_0} \left[ \left( i^\alpha \left[ig_++g_1a_+(1+\alpha)\right]-  (-i)^\alpha \left[ig_- +g_1a_-(1+\alpha)\right]\right)\right]\right\}
  \end{multline}
 \r{assuming the right side of \eqref{eq17} does not vanish.}

    2. Except for values of the parameters $a_+,a_-,g_+,g_-$ that make the leading asymptotic expressions vanish, which occur on a set of {\em zero measure,} in the parameter space, and is a meagre set, the domain of convergence of the SHE is exactly $\{z:|z|\ge R\}$.

  \end{Theorem}
The proof of Theorem \ref{Thm3} is given in \S\ref{ST3}.

\begin{Note}\label{NLowerReg}
As the Theorems\,\ref{Thm1} and \ref{Thm3} show,  for sufficiently regular planets, more regularity results in faster decay of the algebraic part of the asymptotics of the SHE coefficients. Perhaps surprisingly \r{at a glance}, fast decay also occurs in case \eqref{assumeF} as  $\alpha$ decreases. The reason is that, in this case, the tallest peak looks like a thinner and thinner "antenna" contributing with zero mass in the limit $\alpha\to 0$.
\end{Note}

\section{Proof of Theorem\,\ref{Thm1}}\label{PfThm1}
 \subsection{A general lemma.}

\begin{Lemma}\label{L1}

Under the assumptions of Section\,\ref{asnot} the coefficients $C_n$ in \eqref{eq:eqcoef} satisfy
\begin{equation}
    \label{eq:asymp1}
    C_n\asymp\frac{R^{n+3}}{n^{3/2}}\frac{\sqrt{2}}{\sqrt{\pi}}\, \Re \left(e^{-i\pi/4} J\right)
  \end{equation}
where 
\begin{equation}
  \label{eq:defJ}
  J=\int_{I_0} g(\theta)e^{-(n+3) F(\theta)} e^{i(n+1/2)\theta}d\theta
  \end{equation} 
(unless $J=0$, which we show only happens on meagre sets).
\end{Lemma}

\begin{proof}
    
 We substitute $r=r_M e^{-s}$ in \eqref{eq:eqcoef2} to get
\begin{equation}\label{eq:eqcoef3}
  C_n=\int_0^\pi \sin\theta\, P_n(\cos\theta)\, r_M(\theta)^{n+3} d\theta\, \int_0^{\log[r_M(\theta)/r_m(\theta)]}\, e^{-(n+3)s}\, v(r_M e^{-s},\theta)\,ds
\end{equation}
For the innermost integral we apply Watson's Lemma (see e.g. \cite{Ablowitz}, \cite{Book} ). The result is $\asymp \tfrac{g(\theta)}{n+3}\asymp \tfrac{g(\theta)}{n}$. Therefore, using the notation \eqref{eq:defF},
\begin{multline}\label{eq:eqcoef4}
  C_n\asymp \frac{1}{n+3}\, \int_0^\pi \sqrt{\sin\theta}P_n(\cos\theta)r_M(\theta)^{n+3}g(\theta)d\theta\\ \asymp\frac{R^{n+3}}{n} \int_0^\pi \sqrt{\sin\theta}P_n(\cos\theta)\, e^{-(n+3)F(\theta)}g(\theta)d\theta
\end{multline}

Since we assumed $F(\theta)>\delta_1>0$ outside $I_0$ , the part of the integral in \eqref{eq:eqcoef4} in the complement of $I_0$ is exponentially  small. Therefore
\begin{equation}
    \label{star1}
 C_n   \asymp\frac{R^{n+3}}{n} \int_{I_0} \sqrt{\sin\theta}\, P_n(\cos\theta)\, e^{-(n+3)F(\theta)}g(\theta)d\theta
      \end{equation}

Using now the well known asymptotic behavior of Legendre polynomials of large order \cite{DLMF},
\begin{equation}
    \label{star2}
P_{n }(\cos \theta )={\frac {2}{\sqrt {2\pi n \sin \theta }}}\cos \left(\left(n +{\tfrac {1}{2}}\right)\theta -{\frac {\pi }{4}}\right)+{\mathcal {O}}\left(n ^{-3/2}\right),\quad \theta \in (0,\pi )
 \end{equation}
  we get \eqref{eq:asymp1}.
  \end{proof}

 Denote 
 \begin{equation}
    \label{notx}
    x=\theta-\theta_0, \ \ \tilde{F}(x)=F(\theta_0+x),\ \  \tilde{g}(x)=g(\theta_0+x),\ \ \tilde{h}(x)=h(\theta_0+x),\ \ \tilde{\phi}(x)=\phi(\theta_0+x)
    \end{equation}
     ($\tilde{\phi}$ is supported in $[-2\epsilon,2\epsilon]$ and $\tilde{\phi}=1$ on $[-\epsilon,\epsilon]$). With these notations, formula \eqref{eq:defJ} becomes
     \begin{equation}
  \label{eq:defJtilde}
  J=e^{i(n+1/2)\theta_0}\int_{-\delta}^\delta \tilde{g}(x)e^{-(n+3) \tilde{F}(x)} e^{i(n+1/2)x}dx
  \end{equation}

   \begin{Lemma}\label{Genform1}
Under the further assumptions \eqref{Linf} we have 
\begin{equation}
    \label{eq:asymp1}
    C_n\asymp\frac{R^{n+3}}{n^{3/2}}\frac{\sqrt{2}}{\sqrt{\pi}}\Re \left( e^{-i\pi/4} J_n\right)
  \end{equation}
   where
 \begin{equation} \label{eq:asJn0}
 e^{-i(n+1/2)\theta_0} J_n\asymp  \frac 1{\sqrt{2cn}} \int_{-n-n^q}^{-n+n^q}\hat{f}_0(k) e^{-\frac{(k+n)^2}{4cn}}dk -n \frac 1{\sqrt{2cn}} \int_{-n-n^q}^{-n+n^q}\hat{f}_1(k) e^{-\frac{(k+n)^2}{4cn}}dk
\end{equation}
and
 \begin{equation}
 \label{f0f1}
 f_0= \tilde{g}\tilde{\phi},\ \ \  f_1= \tilde{h}\tilde{g}\tilde{\phi}
 \end{equation}
and $q$ is any number such that $\tfrac12<q<1$ (unless $J_n$ vanishes, which we will show is exceptional in Section\,\ref{Thm1p1}).
\end{Lemma}

\begin{proof}

We use Lemma\,\ref{L1}, and further evaluate $J$.
Let $p$ be any number such that $\tfrac 1{\beta}<p<\tfrac 12$.

The part of the integral $J$ in \eqref{eq:defJtilde} over the interval $[n^{-p},\delta]$ is exponentially small:
$$\big| \int_{n^{-p}}^\delta  e^{-n \tilde{F}(x)} e^{inx}  \tilde{g}(x) dx\big|  \le \int_{n^{-p}}^\delta |\tilde{g}(x)| e^{-(n+3) (cx^2 +O(x^\beta))} \le \delta\,  \|g\|e^{- cn^{1-2p}} $$
Similarly the integral over the interval $[-\delta, -n^{-p}]$ is exponentially small. Hence, once we show that the asymptotic behavior of the integral over $[-n^{-p},n^{-p}]$ is power-like, we have, for large $n$
\begin{multline}\label{kapan}
e^{-i(n+\frac12)\theta_0}J  \asymp \int_{-n^{-p}}^{n^{-p}}  e^{-(n+3) \tilde{F}(x)} e^{inx+i\frac x2}  \tilde{g}(x) dx\asymp \int_{-n^{-p}}^{n^{-p}}  e^{-n \tilde{F}(x)} e^{inx}  \tilde{g}(x) dx\\
=\int_{-n^{-p}}^{n^{-p}}  e^{-n \tilde{F}(x)} e^{inx}  \tilde{g}(x) \tilde{\phi}(x)dx 
=\int_{-n^{-p}}^{n^{-p}}  e^{-n (cx^2-ix)} e^{-n\tilde{h}(x)}  \tilde{g}(x) \tilde{\phi}(x)dx\\
\asymp \int_{-n^{-p}}^{n^{-p}}  e^{-n( cx^2-ix)} \left[ 1-n\tilde{h}(x)\right]  \tilde{g}(x) \tilde{\phi}(x)dx 
\asymp \int_{-\infty}^{\infty}  e^{-n( cx^2-ix)} \left[ 1-n\tilde{h}(x)\right]  \tilde{g}(x) \phi(x)dx\\
 = \int_{-\infty}^{\infty}  e^{-n( cx^2-ix)} f_0(x)dx-n\int_{-\infty}^{\infty}  e^{-n( cx^2-ix)} f_1(x)dx
\end{multline}
 where we used the fact that $n\tilde{h}(x)$ is small on the interval of integration (being at most of order $n\cdot n^{-p\beta}=n^{1-p\beta}$) to expand $e^{-n\tilde{h}(x)}$ in series: $e^{-n\tilde{h}(x)}\asymp 1-n\tilde{h}(x)$. Using again the fact that the integral outside $[-n^{-p},n^{-p}]$ is exponentially small, we extend the integral over the full real line ($\tilde{h}\tilde{\phi}$ extends naturally by zero), via \eqref{eq:regcond} and the notation of \eqref{f0f1}.
 
 We have
 $$\int_{-\infty}^{\infty}  e^{-n( cx^2-ix)} f_0(x)dx=\frac 1{\sqrt{2\pi}} \int_{-\infty}^{\infty} e^{-n( cx^2-ix)}dx \int_{-\infty}^\infty e^{ikx}\hat{f}_0(k) dk$$ 
Since $\beta_0,\beta_1>1$, we can interchange the order of integration  to get
 \begin{equation}
 \label{intgs}
 \frac 1{\sqrt{2\pi}} \int_{-\infty}^\infty e^{ikx}\hat{f}_0(k) dk \int_{-\infty}^{\infty} e^{-n( cx^2-ix)+ikx} dx= \frac 1{\sqrt{2\pi}} \int_{-\infty}^\infty \hat{f}_0(k) \sqrt{\frac \pi{cn}}e^{-\frac{(k+n)^2}{4cn}} dk
 \end{equation}
 
 Note that the integral is exponentially small outside an interval $|k+n|<n^q$ since $\tfrac12<q<1$. Therefore  (once we show the main behavior is power-like)
 the last term in \eqref{intgs} is asymptotic to
 $$ \frac 1{\sqrt{2cn}} \int_{-n-n^q}^{-n+n^q} \hat{f}_0(k) e^{-\frac{(k+n)^2}{4cn}} dk $$
 The second integral in \eqref{kapan} is evaluated similarly, leading to \eqref{eq:asJn0}.

 \end{proof}

 \subsection{Proof of Theorem \ref{Thm1}, part 2.}\label{ST1P2}

\begin{Lemma}\label{Genform}
Under the further assumptions \eqref{eq:regcond} $J_n$ in \eqref{eq:asJn0} satisfies
 \begin{equation} \label{eq:asJn}
  J_n\asymp e^{i(n+1/2)\theta_0}\int_{-n^{-p}}^{n^{-p}}  e^{-n( cx^2-ix)} \left[ 1-nh(x)\right]  \tilde{g}(x) \phi(x)dx 
\end{equation}
if $p$ is any number such that $\tfrac 1{\beta}<p<\tfrac 12$.
\end{Lemma}

\begin{proof}
Under the more precise assumptions \eqref{eq:regcond} we can evaluate further the right hand side of \eqref{eq:asJn0} in Lemma\,\ref{Genform1}.

 We have
  \begin{multline}
 \frac 1{\sqrt{2cn}} \int_{-n-n^q}^{-n+n^q} \hat{f}_0(k) e^{-\frac{(k+n)^2}{4cn}} dk = \frac 1{\sqrt{2cn}} \int_{-n^q}^{n^q} \hat{f}_0(\xi-n) e^{-\frac{\xi^2}{4cn}} dk \\
  \asymp \frac {a_0n^{-\beta_0}}{\sqrt{2cn}} \int_{-n^q}^{n^q} e^{-\frac{ \xi^2}{4cn}} dk\asymp \frac {a_0n^{-\beta_0}}{\sqrt{2cn}} \int_{-\infty}^{\infty} e^{-\frac{\xi^2}{4cn}} dk 
 =\frac {a_0n^{-\beta_0}}{\sqrt{2cn}}  2\sqrt{cn\pi}=   \frac{\sqrt{2\pi}a_0 }{n^{\beta_0}}
 \end{multline}
where  we used the fact that, by \eqref{eq:regcond}, $\hat{f}_0(\xi-n)\asymp a_0 n^{-\beta_0}$.

 Similarly
 $$ \frac 1{\sqrt{2cn}} \int_{-n-n^q}^{-n+n^q} \hat{f}_1(k) e^{-\frac{(k+n)^2}{4cn}} dk  \asymp  \frac{\sqrt{2\pi}a_1 }{n^{\beta_1}}$$
 and therefore the integral $J$ in \eqref{eq:defJ} has the behavior
 \begin{equation}
 \label{asyJW}
 J=e^{i(n+1/2)\theta_0}(J_-+J_+)\asymp e^{i(n+1/2)\theta_0}\left(  \frac{\sqrt{2\pi}a_0 }{n^{\beta_0}}-  \frac{\sqrt{2\pi}a_1 }{n^{\beta_1-1}}\right)
 \end{equation}
 unless $\beta_0= \beta_1-1$ and $a_0=a_1$.
 \end{proof}
 
 Relation \eqref{eq:asympt-beh} follows from \eqref{asyJW} and \eqref{eq:asymp1}.

  \subsection{Proof of Theorem\,\ref{Thm1} part 1.}\label{Thm1p1}
  
  Relation \eqref{eq:asympt-beh} implies \eqref{eq:limsup} unless  $\beta_0= \beta_1-1$ and $a_0=a_1$. We will now show that the limit in \eqref{eq:limsup} vanishes only exceptionally, in a sense defined precisely below.

Denote $\phi L^2=\{\phi f\,|\,f\in L^2(\RR)\}$ and $\mathcal{T}=\mathcal{F}(\phi L^2)$, two  closed subspaces of $ L^2(\RR)$. Define $\|F\|_\beta=\sup_{k\in\RR} (1+|k|)^\beta |F(k)|$ for some $\beta>1$. Let $\mathcal{H}_\beta=\{F\in \mathcal{T}\,|\, \|F\|_\beta<\infty\}$. This space is non-null, see section\,\ref{append}. Finally, let $q\in(\tfrac12,1)$ and define $\kappa_n$ on $\mathcal{H}_\beta$ by
$$\kappa_n(F)=\Re \left( e^{i(n+\frac12)\theta_0-i\pi/4} \frac1{\sqrt{2cn}}  \int_{-n-n^q}^{-n+n^q} F(k)\, e^{-\frac{(k+n)^2}{4cn}} dk \right)$$
We have 
\begin{multline}\label{estimkappan}
|\kappa_n(F)|\leq \frac1{\sqrt{2cn}}  \int_{-n-n^q}^{-n+n^q} |F(k)|\, e^{-\frac{(k+n)^2}{4cn}} dk \leq  \frac1{\sqrt{n}} \|F\|_\beta \int_{-n-n^q}^{-n+n^q} \frac 1{(1+|k|)^\beta }\, e^{-\frac{(k+n)^2}{4cn}} dk\\
\leq \frac1{\sqrt{2cn}} \, \|F\|_\beta\, n^{-\beta}\left[ 1+O(n^{-(\beta-q)}\right] \int_{-n-n^q}^{-n+n^q} e^{-\frac{(k+n)^2}{4cn}} dk\leq \sqrt{2\pi}  \|F\|_\beta\, n^{-\beta}\left[ 1+O\left(n^{-(\beta-q)}\right)\right] 
\end{multline}
\begin{Lemma}\label{LKer}
  The linear functionals $L_+^\beta$ and $L_-^\beta$ defined on $\mathcal{H}_\beta$ by 
  $$L_+^\beta(F)=\limsup_{n\to\infty} n^\beta \kappa_n(F),\ \ \ L_-^\beta(F)=\liminf_{n\to\infty} n^\beta \kappa_n(F)$$
  are continuous and nonzero. The closed subspace $K^\beta$ of $\mathcal{H}_\beta$ defined by $K^\beta=\text{Ker}L_+^\beta \cap\text{Ker}L_-^\beta $ is nowhere dense in $\mathcal{H}_\beta$.
\end{Lemma}
\begin{proof}

The estimate \eqref{estimkappan} shows that $L_+^\beta,L_-^\beta$ are continuous on $\mathcal{H}_\beta$. Therefore $\text{Ker}L_+^\beta$ and $\text{Ker}L_-^\beta$ are closed linear subspaces in $\mathcal{H}_\beta$.

To show that $K^\beta:=\text{Ker}L_+^\beta \cap\text{Ker}L_-^\beta$ is a proper subspace, consider the linear subspace of $\mathcal{H}_\beta$ for which 
 $\lim_{k\to -\infty}(-k)^\beta F(k)=a_F$. The same calculation as in the proof of Lemma\,\ref{Genform} shows that if $a_F\ne 0$ then $L_\pm^\beta(F)\ne 0$, hence  $K^\beta$ is a closed proper subspace of $\mathcal{H}_\beta$ completing the proof of Lemma \ref{LKer}.
  \end{proof}

 Consider functions $g,h$ for which  $\limsup |C_n|R^{-n-3}n^{\frac32+\beta_m}=0$, so $\limsup_n C_nR^{-n-3}n^{\frac32+\beta_m}=0=\liminf_n C_nR^{-n-3}n^{\frac32+\beta_m}$.
 
 By Lemma\,\ref{Genform1} we have
 \begin{equation}
    \label{Cn}
    C_n R^{-n-3}n^{\frac 32}\frac{\sqrt{\pi}}{\sqrt{2}}=\kappa_n(\hat{f}_0)-n\kappa_n(\hat{f}_1)+e_n
     \end{equation}
      where $e_n$ is exponentially small.
       Due to \eqref{Linf}, from \eqref{estimkappan} we have $|\kappa_n(\hat{f}_0)|<C_0n^{-\beta_0}$ and $|\kappa_n(\hat{f}_1)|<C_1n^{-\beta_1}$ for some constants $C_0,C_1$.

       If $\beta_0<\beta_1-1$ then $\beta_m=\beta_0$. Multiplying \eqref{Cn} by $n^{\beta_m}$ and taking the $\limsup$, respectively $\liminf$,  we see that $L_+( \hat{f}_0)=L_-(\hat{f}_0)=0$, therefore $\hat{f}_0\in K$. Similarly, if $\beta_0<\beta_1-1$ then $\hat{f}_1\in K$ and if $\beta_0=\beta_1-1$ then $\hat{f}_0-\hat{f}_1\in K$. In any case, this faster decrease of $C_n$ is only possible for functions which, by Lemma \ref{LKer}, form a nowhere dense set.
\vskip.1in
In fact the methods above yield the following stronger topological result.

\begin{Lemma} For any pair $\alpha_2 > \alpha_1 > 1$ the image $im(\iota_{12})$ of the inclusion $\iota_{12}:\mathcal{H}_{\alpha_2}\hookrightarrow\mathcal{H}_{\alpha_1}$ is a closed subspace of $\mathcal{H}_{\alpha_1}$ of infinite codimension.
\end{Lemma}

\begin{proof} For any $\alpha_2 > \alpha >\alpha' > \alpha_1$, the proof of Lemma\,\ref{LKer} above identifies $\mathcal{H}_\alpha$ with a proper closed subspace of the codimension 2 subspace $K^{\alpha'}$ of $\mathcal{H}_{\alpha'}$. Additionally, the short-exact sequence of Banach spaces
\[
0\to K^\gamma\to \mathcal{H}_{\gamma}\to \mathbb R^2\to 0
\]
is topologically split for all $\gamma > 1$. Fix an increasing sequence $\{\gamma_n\}_{n\ge 1}$ with $\alpha_1 < \gamma_1 < \gamma_2 <\dots < \gamma_n <\dots < \alpha_2$. For each $n$ this split exact sequence yields a factorizion
\[
\mathcal{H}_{\gamma_{n+1}}\xrightarrow{i_1}\mathcal{H}_{\gamma_{n+1}}\oplus\mathbb R^2\hookrightarrow K^{\gamma_n}\oplus\mathbb R^2\cong \mathcal{H}_{\gamma_n}
\]
Passing to the inverse limit then yields the factorization of $\iota_{12}$ as
\[
\mathcal{H}_{\alpha_2}\xrightarrow{i_1} \mathcal{H}_{\alpha_2}\oplus\mathbb R^\infty\hookrightarrow
(\varprojlim_n \mathcal{H}_{\gamma_n})\oplus\mathbb R^\infty \hookrightarrow \mathcal{H}_{\alpha_1}
\]
where $\mathbb R^\infty = \varinjlim_n\mathbb R^{2n}$ is equipped with the inductive limit topology.
\end{proof}

 \subsection{Independence of the cut-off function $\phi$}

\begin{Proposition}\label{P6}
For $\epsilon>0$ small enough the estimate \eqref{asyJW}  is independent of the choice of the cutoff function $\phi$ as in the assumptions of Theorem\,\ref{Thm1} and of $\epsilon$.
\end{Proposition}

\begin{proof}

 Indeed, let $\phi_1$ be another cutoff function satisfying assumptions similar to $\phi$ (perhaps with $\epsilon_1$ instead of $\epsilon$). Then, for $-k$ large,
 $$\widehat{\phi_1f}=\hat{\phi_1}*\hat{f}=\int_{-\infty}^\infty \hat{\phi_1}(s)\hat{f}(k-s)\, ds  =\int_{-\sqrt{-k}}^{\sqrt{-k}} \hat{\phi_1}(s)\hat{f}(k-s)\, ds+\int_{|s|>\sqrt{-k}} \hat{\phi_1}(s)\hat{f}(k-s)\, ds$$
 
 Assume $\lim_{k\to -\infty}\hat{f}(k)(-k)^\beta=a$. For large $(-k)$ we have
 $$\int_{-\sqrt{-k}}^{\sqrt{-k}}  \hat{\phi_1}(s)\hat{f}(k-s)\, ds \sim \frac{a}{(-k)^\beta}\int_{-\sqrt{k}}^{\sqrt{k}} \hat{\phi_1}(s)ds \sim \frac{a}{(-k)^\beta}\int_{-\infty}^{\infty} \hat{\phi_1}(s)ds= \frac{a}{(-k)^\beta}\phi(0)= \frac{a}{(-k)^\beta}$$

 Let $N>2\beta+1$. Since $\hat{\phi_1}$ is a Schwartz function, it decays faster than any power, thus $|\phi(s)|<|s|^{-N}$ for  $|s|$ large enough.
 Hence
$$ |\int_{\sqrt{-k}}^\infty\hat{\phi_1}(s)\hat{f}(k-s)\, ds|\leq \|f\|_\infty  \int_{\sqrt{-k}}^\infty s^{-N}ds =\frac{\|f\|_\infty}{N-1}\frac 1{{(-k)}^{\frac{N-1}2}}   =o\left((-k)^{-\beta}\right)$$
Similarly, the integral on $(-\infty,-\sqrt{-k})$ is negligible for large $k$.

Hence, also $\lim_{k\to -\infty}\widehat{f\phi_1}(k)(-k)^\beta=a$.
This shows independence on the choice of $\epsilon$.

\end{proof}

\section{Proof of Theorem\,\ref{Thm2}}\label{ST2}
 
\begin{proof}
   By Abel's theorem, if the power series expansion of $V(z)$ in inverse powers of $z$  converges at some point $z_0> 0$, then $V$ is analytic at infinity, and in $\{z\in\CC:|z|>z_0\}$; in particular, if $z_0<1$ then $V$ is analytic at any point on the unit circle $\{z\in\CC:|z|=1\}$.
    
  We note that analyticity of $Q(p(z))$ in \eqref{eq:reproc1} at any $z\in\CC\setminus\{\pm i\}$ is equivalent to analyticity of $V$ at that point.
  
 Assume the SHE  converges at some point $z_0<1$.

  We use a method of singularity transformation introduced in \cite{CostinDunne}.
Define Laplace-type convolution by 
$$(f*g)(p)=\int_0^p f(s) g(p-s) ds$$
Let $\mathcal{L}$ be the Laplace transform. Noting that $\mathcal{L}(f*g) =\mathcal{L}(f) \mathcal{L}(g)$ it follows immediately that convolution is commutative and associative.
Consider the linear operator $A$ defined by
\begin{equation}
  \label{eq:defop}
  (Af)(p)=\sqrt{p}\frac{d}{dp}\Big[p^{-1/2} *f\Big]
\end{equation}
Let $f$ be a function whose  Maclaurin  series $\sum_{k\ge 0} c_k p^k$ converges in $\DD$.  We claim that $Af$ is also analytic in $\DD$. Since  $\mathcal{L}(p^k*p^{-1/2})=\mathcal{L}(p^k)\mathcal{L}(p^{-1/2})=\sqrt{\pi}\Gamma(k+1)x^{-k-3/2}$, it follows, by taking the inverse Laplace transform, that
$$p^{-1/2} *p^k=\sqrt{\pi}p^{k+1/2}\frac{\Gamma(1+k)}{\Gamma(k+3/2)}$$
Differentiating and using dominated convergence in $\DD$ to integrate term by term, we get 
 \begin{equation}
    \label{eq:eqps}
     (Af)(p)=\sqrt{\pi}\sum_{k=0}^\infty \frac{\Gamma(1+k)}{\Gamma(k+\frac12)}c_k p^k
   \end{equation}
   which is also (manifestly) convergent in $\DD$. Using the binomial formula,
   $$(1-px)^{-1/2}=\sum_{k=0}^\infty \frac{\Gamma \left(k+\frac{1}{2}\right)}{\sqrt{\pi } k!}p^k x^k,\ \ \ |px|<1$$
 it follows that, if $|px|<1$ we have
   $$A((1-px)^{-1/2})=\sum_{k=0}^\infty (px)^k=\frac{1}{1-px}$$
  Let $\mathcal{D}$ be any star-shaped domain in $\CC$ (meaning it contains the origin, and together with any point
   $p_0\in\mathcal{D}$ it contains the straight line segment joining $0$ to $p_0$). Assume $\DD\subset\mathcal D$, and let $f$ be a function analytic in $\mathcal{D}$. Then, $Af$ is also analytic in
   $\mathcal D$. Indeed, changing variable $s=p t$ we have
 $$   (Af)(p)=\sqrt{p}\frac{d}{dp}\int_0^p \frac{f(s)\, ds}{\sqrt{p-s}} =\sqrt{p}\frac{d}{dp}\left[\sqrt{p}   \int_0^1\frac{f(p t)}{\sqrt{1-t}}dt \right]$$
By standard analytic dependence with respect to parameters,  $Af$ is analytic where $f$ is.
 
 Denote $f_x(p)=(1-px)^{-1/2}$. For any $x\in [-1,1]$, $f_x$ is analytic on the star-shaped domain $p^{-1}\notin [-1,1]$. Hence $A(f_x)$ is analytic on the same domain. Since for $|px|<1$ we have $A(f_x)(p)=(1-px)^{-1}$, it means that
$$A(f_x)=\frac{1}{1-px}\ \ \ \ \text{for\, all\ } p\ \text{with\ } p^{-1}\notin [-1,1]$$
Now we apply $A$ to $Q$: for $p^{-1}\notin [-1,1]$ we have
\begin{equation}
  \label{eq:AG}
  (AQ)(p)=\sqrt{p}\frac{d}{dp}\int_0^p \frac{ds}{\sqrt{p-s}}\int_{-1}^1\frac{\mu(x) dx}{\sqrt{1- p x}}
= \int_{-1}^1 \mu(x) A(f_x)(p)\, dx    =    \int_{-1}^1 \frac{\mu(x)dx}{1-p x}
\end{equation}
and, by  analytic dependence on parameters, $AQ$ is
analytic in $p$ if $p^{-1}\notin [-1,1]$. 

On the other hand, if $p^{-1}\in (-1,0)\cup (0,1)$ then $z$ is on the unit circle $\mathbb{T}$ (see \eqref{eq:balayage1}). Since $V$ is analytic in a neighborhood of $\TT$, $Q(p)$ is analytic for $p^{-1}$ in  a neighborhood of $(-1,0)\cup (0,1)$.
  For  $p\ne 0$, with  $\zeta =1/p$ we get
\begin{equation}
 \label{eq:AG1}
  (AQ)(\zeta)=\zeta \int_{-1}^1 \frac{\mu(x)dx}{\zeta-x}
\end{equation}
Analyticity in $\zeta$ 
in  a neighborhood of $(-1,0)\cup(0,1)$ implies the existence of analytic continuation through $(-1,0)\cup(0,1)$ from both the upper and the
lower half plane.
By Plemelj's formulas, see \cite{Ablowitz}, if the limits of $AQ$ on
$(-1,0)\cup(0,1)$ from above and below are  $AQ^+$ and $AQ^-$ resp., then
\begin{equation}
  \label{eq:plemelj1}
  AQ^+(x)-AQ^-(x)=-2\pi i x \mu(x)
\end{equation}
Since $AQ^+$ and $AQ^-$ are analytic in a neighborhood of $(-1,0)\cup(0,1)$, $\mu$ is analytic in a neighborhood of  $(-1,0)\cup(0,1)$, in particular at
$x=\cos\theta_0$.

The converse is proved by standard deformation of the contour of integration, $[-1,1]$. 
\end{proof}

{\bf Note.} The case $\theta_0=0$ would not be special if we had first proved a variation of Plemelj's formulas adapted to a square root kernel. However, we are looking at generic cases, and this would have complicated the proof to cover just one more point.

 \section{Proof of Theorem\,\ref{Thm3}}\label{ST3}
 
 We use the notations \eqref{notx}.

\subsection{Proof of (i): the case $\alpha\in(0,1)$}

Using Lemma\,\ref{L1} we write $J$ in \eqref{eq:defJ} as

\begin{multline}
  \label{eq:defJnew}
  J=e^{i(n+1/2)\theta_0}(J_-+J_+), \ \ \text{where}\\
  J_-=\int_{-\delta}^0 \tilde{g}(x)e^{-(n+3) \tilde{F}(x)} e^{i(n+1/2)x}dx,\ \ \ \ \ \ \ J_+=\int_0^{\delta} \tilde{g}(x)e^{-(n+3) \tilde{F}(x)} e^{i(n+1/2)x}dx
\end{multline}

 For $x>0$ both \eqref{asumgka} and \eqref{asumgkb} have the form $\tilde{g}(x)=cx^{ k}(1+\tilde{h}(x))$ (with $c=g_k$ in the case of \eqref{asumgka}, and $c=g_+$ in the case of \eqref{asumgkb}).
Then 
\begin{equation}
 \label{eq:defJ1}
J_+= c J_{k,+} + o\left(|J_+|\right),\ \ \ \text{where\ \ }J_{k,+}=\int_0^{\delta} x^k e^{-(n+3) \tilde{F}(x)} e^{i(n+1/2)x}dx
\end{equation}

We will show that $J_{k,+}\ne 0$ which will imply that $\tilde{h}$ will not contribute to the leading asymptotics of $J_+$ indeed.

Let $p$ be any number such that 
$$ \max\left\{1, \frac 1{\beta}\right\}<p<\frac 1\alpha$$

We further break the interval $[0,\delta]$ into $[0,n^{-p}]$ and  $[n^{-p},\delta]$. The integral over the latter interval is exponentially small:
$$\big| \int_{n^{-p}}^\delta x^k e^{-(n+3) \tilde{F}(x)} e^{i(n+1/2)x}dx\big|  \le \int_{n^{-p}}^\delta x^k e^{-(n+3) (a_+x^\alpha +O(x^\beta))} \le \delta\,  n^{-pk}e^{- a_+n^{1-p\alpha}} $$

In the integral over $[0,n^{-p}]$, the term $nO(x^{\beta})$ in $n\tilde{F}$  (see \eqref{assumeF}), is small  since $nO(x^{\beta})\to 0$ as $n\to\infty$ and there the exponential $e^{nO(x^\beta)}$ can be expanded in series, $e^{nO(x^\beta)}=1+O(n^{1-p\beta})\asymp 1$. Similarly,   $e^{inx}=1+O(n^{1-p})\asymp 1$ and  $e^{-3a_+x+i/2x}=1+O(n^{-p})$.
Thus we  obtain a power asymptotic behavior, as we claimed:
 \begin{equation}
\label{Jkpg}
J_{k,+} \asymp \int_0^{n^{-p}} x^{k} e^{-n a_+x^\alpha  }dx \asymp {\alpha^{-1} a_+^{-\frac {k+1}\alpha}} \frac {\Gamma(\frac {k+1}\alpha)}{n^{\frac {k+1}\alpha}}
\end{equation}
where the last expression is obtained by changing the variable of integration $a_+x^\alpha=u$, noting that the integral differs from the integral over $[0,\delta]$ by exponentially small terms and then using Watson's Lemma (\cite{Watson}, and for this particular form see Lemma\, 3.37 in \cite{Book}).

For $x<0$ and $g$ of the form \eqref{asumgka} we similarly have
\begin{equation}
 \label{eq:defJ1m}
J_-=g_k J_{k,-}+ o\left(|J_{-}|\right),\ \ \ \text{with\ \ }J_{k,-}=\int_{-\delta}^0 x^k e^{-(n+3) \tilde{F}(x)} e^{i(n+1/2)x}dx
\end{equation}
Once we show that $J_{k,-}\ne 0$, it follows that the contribution of $\tilde{h}$ is of $o(J_-)$.

The substitution  $x\mapsto -x$ brings the calculation of $J_{k,-}$ to the one for $x>0$:
$$J_{k,-}=(-1)^k \int_0^{\delta} x^{k} e^{-(n+3) \tilde{F}(-x)} e^{-i(n+1/2)x}dx$$
Hence, we simply have to substitute $a_+\mapsto a_-$ and $i\mapsto -i$ in \eqref{Jkpg} to obtain
\begin{equation}
\label{Jkpgm}
J_{k,-}\asymp (-1)^k {\alpha^{-1} a_-^{-\frac {k+1}\alpha}} \frac {\Gamma(\frac {k+1}\alpha)}{n^{\frac {k+1}\alpha}}
\end{equation}
which added to \eqref{Jkpg}, multiplying by $c=g_k$, and by the exponential prefactor in \eqref{eq:defJnew}, and the factors in \eqref{eq:asymp1} yields \eqref{eq:asymp10a}.

The calculation for $g$ in the form \eqref{asumgkb} is similar, only the prefactor $(-1)^k$ is missing:
\begin{equation}
\label{Jkpgmb}
J_{k,-}\asymp {\alpha^{-1} a_-^{-\frac {k+1}\alpha}} \frac {\Gamma(\frac {k+1}\alpha)}{n^{\frac {k+1}\alpha}}  
\end{equation}
which multiplied by $g_-$, then added to \eqref{Jkpg} multiplied by $g_+$, then multiplied by the other factors as explained above yields \eqref{eq:asymp10aa1}.

\subsection{Proof of (i): the case $\alpha=1$}

The proof for $\alpha=1$ is similar: let $q$ be a number such that $\beta^{-1}<q<1$ and we break the interval $[0,\delta]$ in \eqref{eq:defJ1} into $[0,n^{-q}]$ and  $[n^{-q},\delta]$. The integral over the latter interval is exponentially small:
$$\bigg| \int_{n^{-q}}^\delta x^k e^{-(n+3) \tilde{F}(x)} e^{i(n+1/2)x}dx\bigg|  \le \int_{n^{-q}}^\delta x^k e^{-(n+3) (a_+x +O(x^\beta))} dx \le \delta\,  n^{-qk}e^{- a_+n^{1-q}} $$

In the integral over $[0,n^{-q}]$, the term $nO(x^{\beta})$ in $n\tilde{F}$  (see \eqref{assumeF}), is small  since $nO(x^{\beta})\to 0$ as $n\to\infty$ and there the exponential $e^{nO(x^\beta)}$ can be expanded in series, $e^{nO(x^\beta)}=1+O(n^{1-q\beta})\asymp 1$. Similarly,  $e^{-3a_+x+i/2x}=1+O(n^{-q})$.
Thus we  obtain a power asymptotic behavior:
\begin{equation}
\label{Jkpgais1}
J_{k} \asymp \int_0^{n^{-q}} x^{k} e^{-n (a_+-i)x  }dx \asymp  a_+^{-(k+1)} \frac {\Gamma(k+1)}{n^{k+1}}
\end{equation}
where the last expression is obtained by changing the variable of integration $(a_+-i)x=u$, noting that the integral differs from the integral over $[0,\delta]$ by exponentially small terms and then using Watson's Lemma.

The rest of the details are similar to the case $\alpha\in(0,1)$.

\subsection{Proof of (ii): the case $\alpha\in(1,2)$}

 We write $J$ in \eqref{eq:defJ} as

\begin{equation}
  \label{Jcase3}
  J=e^{i(n+1/2)\theta_0}J_1, \  \ \ \  \text{where\ } J_1=\int_{-\delta}^\delta e^{-n[\tilde{F}(x)-ix]} f(x)\, dx,\ \ \ \ 
  f(x)= \tilde{g}(x)e^{-3 \tilde{F}(x)} e^{ix/2}dx
\end{equation}

We have 
$$\int_0^\delta e^{-n[\tilde{F}(x)-ix]} f(x)\, dx= \int_0^\delta e^{-n[a_+x^\alpha-ix]} f(x)\, dx\asymp  \int_0^\infty e^{-n[a_+x^\alpha-ix]} f(x)\, dx$$
where the last relation holds since the two integrals differ by exponentially small terms and we show below that the last integral has power behavior.

Indeed, we have
$$ \int_0^\infty e^{-n[a_+x^\alpha-ix]} f(x)\, dx=\int_\mathcal{C} e^{-nu} \frac{f(x(u))}{\alpha a_+x^{\alpha-1}-i}\, du=\int_0^\infty e^{-nu} \frac{f(x(u))}{\alpha a_+x^{\alpha-1}-i}\, du$$
where we changed the variable of integration to $u=a_+x^\alpha-ix$ and $\mathcal{C}$ is a path in the fourth quadrant stating at the origin; since the integrand is singular only on the positive imaginary axis, the path of integration can be pushed along $\RR_+$.

As $u\to 0$ we have $x\sim iu+ia_+i^\alpha u^\alpha$ and therefore 
$$\frac{f(x(u))}{\alpha a_+x(u)^{\alpha-1}-i}\sim -g_1u  +i^\alpha u^\alpha\left[ig_++g_1a_+(1+\alpha)\right]\ \ \ (u\to 0)$$
and by Watson's Lemma
$$\int_0^\infty e^{-nu} \frac{f(x(u))}{\alpha a_+x(u)^{\alpha-1}-i}\, du \asymp -\frac{g_1}{n}+ i^\alpha \left[ig_++g_1a_+(1+\alpha)\right]\frac{\Gamma(\alpha+1)}{n^{\alpha+1}}$$

To evaluate $\int_{-\delta }^0e^{-n[\tilde{F}(x)-ix]} f(x)\, dx$, after changing the variable of integration $x\to -x$ we see that we have the same integral as in the previous case, only with $a_+$ replace by $a_-$, $g_+$ by $g_-$,  $i$ by $-i$ and $g_1$ by $-g_1$. Adding the asymptotic behavior of the two integrals we obtain

\begin{equation}
  \label{asy2}
  \int_{-\delta }^\delta e^{-n[\tilde{F}(x)-ix]} f(x)\, dx\asymp \frac{\Gamma(\alpha+1)}{n^{\alpha+1}}\left( i^\alpha \left[ig_++g_1a_+(1+\alpha)\right]+  (-i)^\alpha \left[-ig_- -g_1a_-(1+\alpha)\right]\right)
  \end{equation}

\subsection{Proof of (ii) when  $\alpha=2$}

The proof is similar to the previous case, only now, for $x>0$,
$$\frac{f(x(u))}{\alpha a_+x(u)^{\alpha-1}-i}\sim -{g_1}\,u+ \left( -i g_+-3\,{a_+}\,{ g_1}+{g_1}/2
 \right) {u}^{2}\ \ (u\to 0)$$
and by Watson's Lemma
$$\int_0^\infty e^{-nu} \frac{f(x(u))}{\alpha a_+x(u)^{\alpha-1}-i}\, du \asymp -\frac{g_1}{n}+  \left( -ig_+-3g_1a_+ +g_1/2\right)\frac{2}{n^{3}}$$
As above, the asymptotic behavior for the integral with negative $x$ is the same, only with $a_+$ replace by $a_-$, $g_+$ by $g_-$,  $i$ by $-i$ and $g_1$ by $-g_1$
yielding
\begin{equation}
  \label{asy22}
  \int_{-\delta }^\delta e^{-n[\tilde{F}(x)-ix]} f(x)\, dx\asymp \frac{2}{n^{3}}\left( -ig_+-3g_1a_+    +ig_-+3g_1a_-    \right)
  \end{equation}

  %%%%%%%%%%%%%%%%%%%%%%%%%%%%%%%%%%%%%%%%%%%%%%
  %%%%%%%%%%%%%%%%%%%%%%%%%%%%%%%%%%%%%%%%%%%%%%

\section{Further discussions}\label{Further}
\subsection{Connection between regularity and the behavior of the Fourier transform}
The  smoothness of $f$ is characterized by how fast $\hat{f}(k)$ goes to zero as $|k|\to \infty$. To illustrate this, assume $f$ has $n$ derivatives in $L^1(\RR)$. Then, by $n$ integrations by parts we get $\hat{f}=i^n k^{-n} \widehat{f^{(n)}}$, hence  $\hat{f}(k)$ goes to zero faster than $|k|^{-n}$ as $k\to \infty$. We see that, in Fourier space, $n$ can be replaced by any positive number, and this gives a finer characterization of regularity.

\begin{figure}
  \centering\includegraphics[scale=0.4]{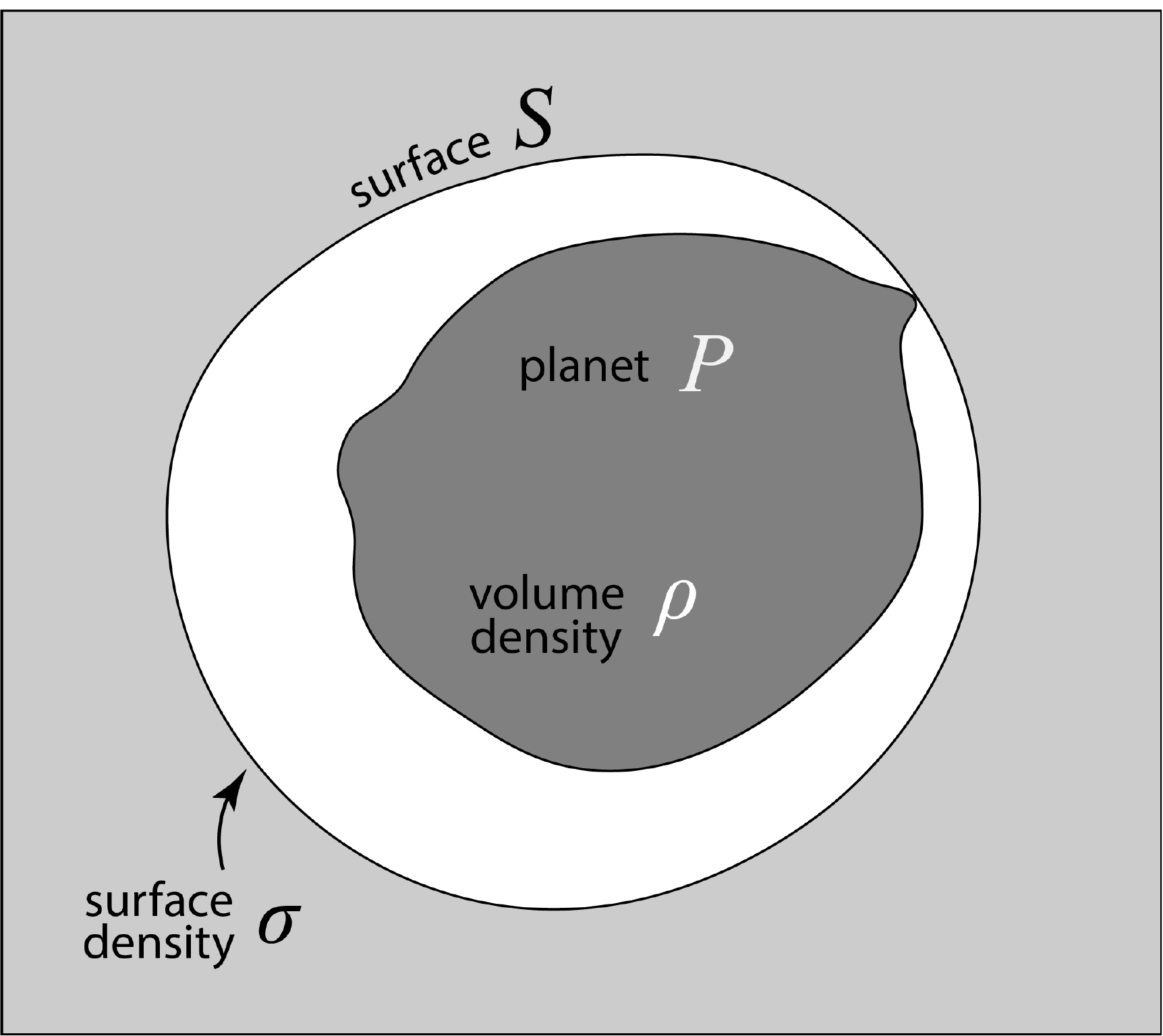}
  \caption{The setting for the physical argument. The light grey region is initially filled with a conductor. The surface depicted here touches the planet at one point only. For the purpose of proving the theorem, we replace the planet $P$ with a planet $P'$ which {\em extends all the way to $S$} by looking at the empty space between $P$ and $S$ as a part of $P'$ where the density is zero.\label{F2}}
  \end{figure}
\subsection{A physics proof of balayage}\label{S:phys1}

The balayage theorem, by now well known in potential theory, electrostatics and gravitation theory was  first proved by Poincar\'e, \cite{bal1,bal2,bal3,bal4,bal5}. In essence, it says the following: given a volume mass density function $\rho$ of a planet $P$, bounded by the surface  $\partial P=S$, there is a surface  density function $\sigma$ on $S$ which produces the same potential in the domain exterior to $S$; $\sigma$ coincides with the harmonic measure on $S$.  Though it admits a short physical ``proof''\footnote{It appears that Poincar\'e himself was well aware of some simple physics proof. Indeed,  on p. 5 of \cite{bal1}  he notes that such a statement would be beyond doubt for physicists.}, presented below,  it is remarkable that there seems to be no comparably short mathematical proof.

In our concrete example $S$ is the Brillouin sphere, which touches the planet $P$ at just one point, and is not the boundary of $P$.

To bring this setting to the standard one of the balayage theorem,  we replace $P$ with a planet $P'$ which {\em extends all the way to $S$} by looking at the empty space between $P$ and $S$ as a part of $P'$ where the density is zero. There is of course no difference between $P$ and $P'$ in terms of gravitational potential.

For the   physics ``proof'', it is  easier to do it as a problem in electrostatics, relying on the fact that the governing mathematical equations are exactly the same. In that language, the function  $\rho$ is a nonnegative charge density. For a physically realistic situation, we imagine the whole volume $P'$ as a perfect insulator (otherwise over time the whole charge would migrate to the surface).
 
Imagine that we filled the space outside  $S$ with a conductor (see Figure \ref{F2}). The positive charge of $P'$ will attract electrons from the conductor (outside $S$)  toward it.

At equilibrium,  the electric field in the interior of the conductor  clearly must be zero. This also implies (by Gauss' flux law) that the charge of any domain strictly inside the conductor must be zero, hence all the negative charge from the conductor must go to the surface $S$, where it will have some (non-positive) {\em surface} charge density function $\sigma$.\footnote{See also Feynman's lectures on Physics, \cite{Feynman}.} Now, since there is no net charge anywhere strictly inside the conductor,  it cannot contribute to the potential (which is zero there)  and, at this stage the conductor can be simply removed without change the potential in the region it occupied. It follows that the potential of $S$ with the surface density function $\sigma$ cancels the potential of $P'$ outside $S$. Hence in the absence of $P'$, a charge density function $-\sigma$ on $S$ creates the same potential, outside $S$, as the planet which proves the statement.

A calculation shows that $\sigma$ is given, in terms of the normal derivative of the Green's function of $P'$, by the formula
\begin{equation}
  \label{eq:formulasigma}
 \sigma(y)= \int_{P'}\frac{\partial G(x,y)}{\partial n_y}\rho(x)dx;\ y\in S
\end{equation}
where $G$ is the Green's function of the domain $P'$.

\subsection{The Green's function and the non-physicality of analyticity of the harmonic measure}
\label{S:Greenf}
As mentioned, in our use of the balayage theorem, $S$ is the Brillouin sphere, which touches Earth ($P$) at just one point.
For a sphere, the Green's function is elementary (see \cite{Evans}), and can be calculated by the method of images. Specifically, if we normalize the radius of the sphere to $1$, then
\begin{equation}\label{eq:GF}
  G(x,y)=\Phi(y-x)-\Phi(|x|(y-x^*))
\end{equation}
where $x^*$ is the dual point of $x$, obtained by reflection across the sphere,
\begin{equation}
  \label{eq:dualp}
  x^*=\frac{x}{|x|^2}
\end{equation}
and $\Phi$, in $n\ge 3$ dimensions, is given by
\begin{equation}\label{eq:GFE}
 \Phi(x)=\frac{C(n)}{|y|^{n-2}}
\end{equation}
where the constant $C$ depends on the dimension only, and has no bearing to the arguments. It is straightforward to show that $\sigma(y)$ is analytic in the exterior of $P'$ as well as in the region between $P$ and $S$ where $\rho=0$.

This means that there is just one point on $P'$ which matters for analyticity, namely the point of contact with $S$. Returning to the electrostatics analogy, it is known that a cusp at the point of contact would trigger an infinite electric field (cf.  St. Elmo's fire). More generally, a singularity in a higher derivative of the relevant features of $P$ at the highest point would result in a similar singularity in the corresponding derivative of $E$. Thus, the condition of analyticity at $E$ is ``not to be expected'' from a real celestial body.

\subsection{Summary of the results}

In this section, the SHE of a potential $V$ is denoted by $\textit{SHE}(V)$. For a planet $P$, it is known that the spherical harmonic series $\textit{SHE}(V)$ of its gravitational potential $V$ {\it can} converge all the way down to the topography of $P$ even for highly non-spherical topographies (see \cite{ocb} for elementary examples of this). However, how often does this occur? More precisely, for a {\it generic} planet, how often does this happen?
\vskip.1in

Our paper gives a statistically definitive answer to this question, not only for the Earth but for nearly all planets and all possible mass-density functions. The following summarizes the first set of results, as stated in Theorem 1.
\vskip.1in

{\bf Result 1} {\sf For a generic planet $P$ (as defined above) with gravitational potential function $V_P$, the event that $\textit{SHE}(V_P)$ converges anywhere below the Brillouin sphere of $P$ occurs with probability zero. More precisely, within a natural Banach space $S$ realizing a particular degree of regularity within a small neighborhood of the tip of the highest peak, the subspace $S_a$ of mass-density functions $\sigma$ yielding a gravitational potential $V_\sigma$ for which $\textit{SHE}(V_\sigma)$ converges anywhere below the Brillouin sphere is both
\begin{itemize}
\item a meagre set (of first Baire category) of $S$, and
\item a subspace of infinite codimension 
\end{itemize}}

This last property implies, among other things, that $S_a$ is nowhere dense in $S$ (although it is a much stronger statement).
\vskip.2in
The framework for Theorem 1 involves 3-dim.~mass-density functions, or 3-dim.~measures. Any such measure on $P$ may be ``swept" to the boundary (topography) $\partial P$ of $P$ using the balayage technique developed by Poincar\'e \cite{bal1,bal2,bal3}. The advantage of doing this is that it provides a context in which we identify {\it necessary and sufficient} conditions for $\textit{SHE}(V)$ to converge below the Brillouin sphere. Referring to the definition of $\mu(\cos(\theta))$ given in (\ref{eq:oneint}), the second theorem states
\vskip.1in

{\bf Result 2}  {\sf For a generic planet $P$ with potential $V_P$, $\textit{SHE}(V_P)$ converges at some point below the Brillouin sphere if and only $\mu$ is real-analytic on $(0,\pi/2)\cup (\pi/2,\pi)$.}
\vskip.1in
 
Given that such analyticity occurs with zero probability, this result yields an alternative proof of the statistical solution to the convergence question given above.
\vskip.2in

\subsection{Some remarks on the modelling of a planet's gravitational field} In the same way the Brillouin sphere $S_{Br}(P)$ for a planet $P$ is the smallest sphere centered at $\bf 0$ containing $P$, the Bjerhammer sphere $S_{Bj}(P)$ of $P$ is defined as the largest sphere centered at $\bf 0$ contained within $P$. A well-known result in mathematical geodesy states that the gravitational potential $V$ in the free space $P^c := \mathbb R^3\backslash P$ exterior to the planet's surface can be realized - in the region $P^c$ - as the limit of a sequence of functions $\{V_n\}$ harmonic on the region $S_{Bj}(P)^e$ of $\mathbb R^3$ exterior to $S_{Bj}(P)$, with $\{V_n\}_{n\ge 1}$ converging uniformly to $V$ on any closed subset of $P^c$.
\vskip.1in

For a potential $V'$, write $SHE(V')$ for the spherical harmonic series of $V$ about infinity, and $SHE_k(V')$ for the truncation of the series $SHE(V')$ in the radial coordinate at degree $-k$. As $V_n$ is harmonic, its spherical harmonic expansion $\textit{SHE}(V_n)$ about infinity converges uniformly to $V_n$ on all of $S_{Bj}(P)^e$ (hence $\displaystyle{\lim_{m\to \infty}} \textit{SHE}_m(V_n) = \textit{SHE}(V_n)$). Thus, for any closed subset $K$ of $P^c$,
\begin{itemize}
\item[(C1)] $\{\textit{SHE}_m(V_n)\}$ converges uniformly to $V_n$ on $K$; 
\item[(C2)] $\{V_n\}$ converges uniformly to $V$ on $K$.
\end{itemize}
At first glance these points seem to imply that the technique of spherical harmonic expansions provides everything needed to estimate $V$ to any degree of accuracy and as close to the topography as we would like. And as a purely theoretical statement, this is true.
\vskip.1in
However, in terms of providing a practical method of computation, it is of little use. There are a couple of reasons for this. The first is that in order to constructively (rather than just existentially) prove the existence of the harmonic functions $V_n$, one needs exact knowledge of $V$ on the topography of $P$, something that can never be achieved in practice. Secondly, the method of proof has nothing to say about the rate of convergence. In other words, there are no empirical methods known for determining, for a given planet, how far out one would have to go in both coordinates in order to achieve a desired degree of approximation to $V$.
\vskip.1in
The third point, however, regards $V$ itself. For the above picture has, on occasion, been misinterpreted to mean that the convergence of the spherical harmonic expansions $\textit{SHE}(V_n)$, for each $n$ individually, can have bearing on the convergence of $\textit{SHE}(V)$. The confusion here originates with the well-known failure (in general) of the commutation of bi-graded limits. The method of computation of the spherical harmonic coefficients implies that for all $m,n\ge 0$ one has
\[
\lim_{n\to \infty} \textit{SHE}_m(V_n) = \textit{SHE}_m(V)
\]
Given this, the following is a direct consequence of the results of this paper
\vskip.1in

{\bf\underbar{Failure of commutation of limits}}\, {\it For a generic planet $P$ generating a gravitational potential $V$, and for any sequence of harmonic functions $\{V_n\}$ with domain $S_{Bj}^e$ converging to $V$ in the manner described above, the inequality
\begin{equation}\label{eqn:ne}
\lim_{m\to\infty}\lim_{n\to\infty}\textit{SHE}_m(V_n)\} \ne \lim_{n\to\infty}\lim_{m\to\infty}\textit{SHE}_m(V_n)\}
\end{equation}
holds with probability one below the Brillouin sphere. In fact, equality holds between the two sides of (\ref{eqn:ne}) precisely when $\textit{SHE}(V)$ converges everywhere below the Brillouin sphere.}
\vskip.1in

It is additionally worth noting that uniform $C^0$ or even $C^\infty$ convergence of real-analytic functions on their common domain implies nothing about the real-analyticity of the limit. The following elementary example illustrates this point.
\vskip.1in

{\bf\underbar{Example}} {\sf Let $f:\mathbb R\to \mathbb R$ be a smooth ($C^\infty$) function that is nowhere analytic (c.f. \cite{na}). By the result of Carleman \cite{tc}, we may construct a sequence of functions $\{f_n\}_{n\in\NN}$  holomorphic in $\CC$ uniformly  converging to $f$ on its entire domain $\mathbb R$. In fact, given the smoothness of $f$, we can, for each finite $k$,  by integrating $k$ times, choose the sequence $\{f_n\}$ to converge uniformly to $f$ on $\mathbb R$ in the Banach $C^k$-norm. The result is then a sequence of functions $\{f_n\}_{n\in\NN}$ holomorphic in $\CC$ converging to $f$ uniformly in the $C^\infty$ Fr\'echet topology.}

\section{Appendix}\label{append}

The space $\mathcal{H}_\beta$ in Section\,\ref{Thm1p1} is not null.
 
 Indeed, for $\beta>1$ not an integer, consider for example the Fourier transform $G$ of $|x|^{\beta-1}P\phi$ where $P$ is a polynomial such that $P^{(j)}(\pm\epsilon)=0$ for $j=0,1,\ldots,m$ with $m>\beta$ and such that $P(0)=1$: 
    \begin{equation}
      \label{formG}
      G(k)=\int_{-\infty}^\infty e^{-ikx} |x|^{\beta-1}P(x)\phi(x)\, dx= \int_{-\epsilon}^\epsilon e^{-ikx} |x|^{\beta-1}P(x)\, dx+\int_{\epsilon<|x|<2\epsilon }e^{-ikx} |x|^{\beta-1}P(x)\phi(x)\, dx
      \end{equation}
 The part of the first integral over $[0,\epsilon]$ can have its integration path deformed in the lower half plane, after which we use Watson's Lemma:
\begin{multline}
\int_0^\epsilon e^{-ikx} x^{\beta-1}P(x)\, dx =\int_0^{-i\infty} e^{-ikx} x^{\beta-1}P(x)\, dx+\int_{\epsilon-i\infty}^{\epsilon} e^{-ikx} x^{\beta-1}P(x)\, dx\\
=(-i)^\beta \int_0^{\infty} e^{-kt} t^{\beta-1}P(-it)\, dt- (-i)^\beta e^{-ik\epsilon}\int_0^{\infty} e^{-kt} (t+i\epsilon)^{\beta-1}P(\epsilon-it)\, dt\\
\asymp (-i)^\beta\frac{\Gamma(\beta)}{k^\beta}P(0)- (-i)^\beta e^{-ik\epsilon} \frac{\Gamma(\beta)}{k^\beta}P(\epsilon)= (-i)^\beta\frac{\Gamma(\beta)}{k^\beta}
\end{multline}

Similarly, for $x\in[-\epsilon,0]$
\begin{multline}
\int_{-\epsilon}^0 e^{-ikx} |x|^{\beta-1}P(x)\, dx=e^{\frac{3i\pi}2(\beta-1)}\int_{-\epsilon}^{-\epsilon-i\infty}e^{-ikx} x^{\beta-1}P(x)\, dx+e^{\frac{3i\pi}2(\beta-1)}\int_{-i\infty}^0 e^{-ikx} x^{\beta-1}P(x)\, dx\\
=e^{\frac{3i\pi}2(\beta-1)}e^{ik\epsilon}(-i)^\beta \int_{0}^{\infty}e^{-kt} (t+i\epsilon)^{\beta-1}P(-\epsilon-it)\, dt- e^{\frac{3i\pi}2(\beta-1)} (-i)^\beta \int_{0}^{\infty}e^{-kt} t^{\beta-1}P(-it)\, dt\\
\asymp  -ie^{\pi i \beta} \frac{\Gamma(\beta)}{k^\beta}
\end{multline}

The part of the second integral in \eqref{formG} over $[\epsilon,2\epsilon]$ is integrated by parts $m$ times, after which it becomes $O(k^{-m})$, hence it is much smaller than $k^{-\beta}$. Similarly, the part of the integral over $[-2\epsilon,\epsilon]$  it is much smaller than $k^{-\beta}$.

For integer $\beta$ a similar formula with a jump discontinuity in the $\beta$'th derivative at $0$ yields similar results.

\end{document}